\documentclass[12pt]{amsart}

\usepackage[english]{babel}
\usepackage{mathrsfs,amssymb}
\usepackage{mathtools}
\usepackage[colorlinks, citecolor = blue]{hyperref}

\usepackage[shortlabels]{enumitem}
\setlist[itemize]{leftmargin=25pt}
\setlist[enumerate]{leftmargin=25pt}

\newtheorem{theorem}{Theorem}[section]
\newtheorem{lemma}[theorem]{Lemma}
\newtheorem{prop}[theorem]{Proposition}
\newtheorem{cor}[theorem]{Corollary}

\theoremstyle{definition}
\newtheorem{definition}[theorem]{Definition}
\newtheorem{que}[theorem]{Question}

\theoremstyle{remark}
\newtheorem{remark}[theorem]{Remark}

\newtheorem{example}[theorem]{Example}

\numberwithin{equation}{section}
\usepackage[colorinlistoftodos,prependcaption,textsize=small]{todonotes}

\DeclareMathOperator*{\esssup}{ess\,sup}
\DeclareMathOperator*{\essinf}{ess\,inf}

\DeclareMathOperator{\loc}{loc}

\DeclareMathOperator{\BMO}{BMO}

\newcommand{\N}{\ensuremath{\mathbb{N}}}

\newcommand{\R}{\ensuremath{\mathbb{R}}}

\newcommand{\mc}{\mathcal}

\DeclarePairedDelimiter\abs{\lvert}{\rvert}

\DeclarePairedDelimiter\cbrace\{\}
\DeclarePairedDelimiter\ha()
\DeclarePairedDelimiter{\ip}\langle\rangle
\DeclarePairedDelimiter{\nrm}\lVert\rVert

\newcommand{\nrmb}[1]{\bigl\|#1\bigr\|}
\newcommand{\absb}[1]{\bigl|#1\bigr|}
\newcommand{\hab}[1]{\bigl(#1\bigr)}

\newcommand{\nrms}[1]{\Bigl\|#1\Bigr\|}

\newcommand{\has}[1]{\Bigl(#1\Bigr)}
\newcommand{\cbraces}[1]{\Bigl\{#1\Bigr\}}

\newcommand{\dd}{\hspace{2pt}\mathrm{d}}
\newcommand{\ddn}{\mathrm{d}}
\newcommand{\ee}{\mathrm{e}}

\newcommand{\A}[1]{\mc{A}_{#1}}

\let \la=\lambda
\let \e=\varepsilon
\let \d=\delta

\let \a=\alpha
\let \f=\varphi

\let \O=\Omega

\let \ga=\gamma

\allowdisplaybreaks

\begin{document}
\title[BMO with respect to Banach function spaces]
{BMO with respect to Banach function spaces}

\author[A.K. Lerner]{Andrei K. Lerner}
\address[A.K. Lerner]{Department of Mathematics,
Bar-Ilan University, 5290002 Ramat Gan, Israel}
\email{lernera@math.biu.ac.il}

\author[E. Lorist]{Emiel Lorist}
\address[E. Lorist]{Department of Mathematics and Statistics\\ University of Helsinki \\ P.O. Box 68\\
FI-00014 Helsinki\\ Finland}
\email{emiellorist@gmail.com}

\author[S. Ombrosi]{Sheldy Ombrosi}
\address[S. Ombrosi]{Departamento de Matem\'atica\\
Universidad Nacional del Sur\\
Bah\'ia Blanca, 8000, Argentina}\email{sombrosi@uns.edu.ar}

\thanks{The first author was supported by ISF grant no. 1035/21. The second author was supported by the Academy of Finland through grant no. 336323. The third author was partially supported
by ANPCyT PICT 2018-2501.}

\begin{abstract}
For every cube $Q \subset \R^n$ we let $X_Q$ be a quasi-Banach function space over $Q$ such that $\nrm{\chi_Q}_{X_Q} \simeq 1$, and for $X= \cbrace{X_Q}$ define
\begin{align*}
\nrm{f}_{\BMO_X} &:=\sup_Q \,\nrm{f-{\textstyle\frac{1}{|Q|}\int_Qf} }_{X_Q},\\
\nrm{f}_{\BMO_X^*} &:=\sup_Q \,\inf_c\, \nrm{f-c}_{X_Q}.
\end{align*}
We study necessary and sufficient conditions on $X$ such that
$$
\BMO = \BMO_X = \BMO_{X}^*.
$$
In particular, we give a full characterization of the embedding $\BMO \hookrightarrow \BMO_X$ in terms of so-called sparse collections of cubes and we give easily checkable and rather weak sufficient conditions for the embedding $\BMO_X^* \hookrightarrow \BMO$. Our main theorems recover and improve all previously known results in this area.
\end{abstract}




\maketitle

\section{Introduction}
The space of functions of bounded mean oscillation, abbreviated BMO, was introduced by John--Nirenberg \cite{JN61}. Given a locally integrable function $f\colon \R^n \to \R$, we say that $f\in \BMO$ if
$$\|f\|_{\BMO}:=\sup_{Q}\frac{1}{|Q|}\int_Q|f(x)-\langle f\rangle_Q|\dd x <\infty,$$
 where the supremum is taken over all cubes $Q\subset {\mathbb R}^n$ and we define $\langle f\rangle_Q:=\frac{1}{|Q|}\int_Qf$.
An important and frequently used fact about $\BMO$ is that it is self-improving in the sense that for every $p\in (0,\infty)$,
\begin{equation}\label{si}
\|f\|_{\BMO}\simeq\sup_{Q}\has{\frac{1}{|Q|}\int_Q|f(x)-\langle f\rangle_Q|^p \dd x}^{1/p}.
\end{equation}
This is a consequence of the John--Nirenberg inequality \cite{JN61} and the  John--Str\"omberg theorem \cite{J65, S79}.

The goal of this paper is to generalize (\ref{si}) by replacing the normalized $L^p$-norm on the right-hand side by a more general function space norm and characterize those spaces for which such an extension is possible.

The following approach to this problem has been considered in, e.g., \cite{H09, H12, I16, INS19}. Suppose that $X$ is a (quasi-)Banach function space over ${\mathbb R}^n$.
Then one can define $\BMO_X$ as the space of all locally integrable $f\colon \R^n \to \R$ such that
\begin{equation}\label{defbmox}
\|f\|_{\BMO_X}:=\sup_Q\frac{\|(f-\langle f\rangle_Q)\chi_Q\|_{X}}{\|\chi_Q\|_X}<\infty.
\end{equation}
Using this notation, \eqref{si} reads as $\|f\|_{\BMO}\simeq \|f\|_{\BMO_{L^p}}$,
and one aims to replace $L^p$ by more general $X$.
The main result of \cite{H09, H12, I16, INS19} is that  if $X$ is a Banach function space and
the maximal operator is bounded on the associate space $X'$, then $\|f\|_{\BMO}\simeq \|f\|_{\BMO_{X}}$.

Defining the $\BMO_X$-norm by (\ref{defbmox}) has the drawback that it is strongly adapted to the case when $X=L^p$. When $X$ is the Orlicz space $\f(L)$, a more natural way to define the corresponding space
$\BMO_{\f}$ is by using the normalized Luxemburg norm
$$
\nrm{f}_{\f,Q}:=  \inf \cbraces{\alpha>0 : \frac{1}{\abs{Q}} \int_Q \varphi\has{\frac{\abs{f}}{\a}}\dd x \leq 1}
$$
and defining $\BMO_{\f}$ as the space of all locally integrable $f\colon \R^n \to \R$ such that
\begin{equation*}
\|f\|_{\BMO_{\f}}:=\sup_Q\|f-\langle f\rangle_Q\|_{\f,Q}<\infty.
\end{equation*}
Moreover, the a priori local integrability assumption may not be natural for certain $\varphi$. To alleviate this problem one can define $\BMO_\varphi^*$ as the space of all $f$ such that
\begin{equation*}
\|f\|_{\BMO_{\f}^*}:=\sup_Q \inf_c \|f-c\|_{\f,Q}<\infty,
\end{equation*}
where the infimum is taken over all scalars $c$.  This approach has been used in a recent paper \cite{CPR20}, which in turn was motivated by \cite{LSSVZ15}. In particular, the estimate $\|f\|_{\BMO}\lesssim \|f\|_{\BMO_{\f}^*}$ was shown in \cite{CPR20}
under rather general assumptions on $\f$.

In this paper we unify both approaches.
Let $X_Q$ be a (quasi)-Banach function space over $Q$ for every cube $Q\subset {\mathbb R}^n$. We call $X = \cbrace{X_Q}$ a \emph{family of normalized (quasi)-Banach function spaces} if, for every cube $Q \subset \R^n$, we have
$
\nrm{\chi_Q}_{X_Q}\simeq 1.
$ For example, if
$$
\|f\|_{X_Q}:=\frac{\|f\chi_Q\|_{X}}{\|\chi_Q\|_X},
$$
then $\|\chi_Q\|_{X_Q}=1$ and if $\|f\|_{X_Q}:=\|f\|_{\f,Q}$, then $\|\chi_Q\|_{X_Q}=1/\f^{-1}(1)$.
Given a family of  normalized (quasi)-Banach function spaces  $X = \cbrace{X_Q}$, we define  $\BMO_X$  as the space of all locally integrable $f\colon \R^n \to \R$ such that
$$
\nrm{f}_{\BMO_X} :=\sup_Q \,\nrm{f-\ip{f}_Q}_{X_Q}<\infty,
$$
and $\BMO_X^*$ as the space of all $f$ such that
$$
\nrm{f}_{\BMO_X^*} :=\sup_Q \,\inf_c\, \nrm{f-c}_{X_Q}<\infty.
$$
We of course trivially have $\BMO_X \hookrightarrow  \BMO_X^*$ and in the classical setting the converse holds as well. The
John--Nirenberg inequality \cite{JN61} can be rephrased as $\|f\|_{\BMO}\simeq \|f\|_{\BMO_{\exp L}}$, where we denote
$$\exp L(Q):=\{f:\|f\|_{\f,Q}<\infty\}.$$
for $\varphi(t) = \ee^t-1$

Our first main result is a complete characterization of the embedding $\BMO\hookrightarrow \BMO_X$, which implies the embedding $\BMO\hookrightarrow \BMO_X^*$. Since any function in $\BMO$ is locally integrable, it is natural to study the stronger embedding involving $\BMO_X$.

\begin{theorem}\label{mr}
Let $X = \cbrace{X_Q}$ be a family of normalized quasi-Banach function spaces. The following statements are equivalent:
\begin{enumerate}[(i)]
\item $\BMO \hookrightarrow \BMO_X$.
\item \label{it:mr:2} For any cube $Q$ and sequence of nested sets $\O_k\subset Q$ satisfying $|\O_k|\le 2^{-k}|Q|$ we have
\begin{equation*}
\nrms{\sum_{k=0}^{\infty}\chi_{\O_k}}_{X_Q}\lesssim 1.
\end{equation*}
\item For any cube $Q$ we have $\exp L(Q) \hookrightarrow X_Q$.
\end{enumerate}
\end{theorem}
We will prove a quantitative and more general version of Theorem~\ref{mr} in Section \ref{sec:main}. In particular, we will show that the statements of Theorem~\ref{mr} are also equivalent to a norm estimate for so-called \emph{sparse} families of dyadic cubes.

\bigskip

For the converse embedding, it is, as already mentioned, not natural to a priori assume local integrability. Therefore, we will study the embedding $\BMO_X^*\hookrightarrow \BMO$. Whenever this embedding holds, we can of course conclude that any $f \in \BMO_X^*$ is locally integrable and moreover we have $\BMO_X \hookrightarrow \BMO$.

We do not obtain a full characterization of the embedding $\BMO_X^*\hookrightarrow \BMO$. However, we do obtain very weak sufficient conditions, which recover several results in the literature.

\begin{theorem}\label{jscor} Let $X=\{X_Q\}$ be a family of normalized quasi-Banach function spaces. Suppose either of the following conditions:
\begin{enumerate}[(i)]
\item\label{it:jscor:1} For any
cube $Q$ and $E\subset Q$ with $|E|\ge \frac12|Q|$, we have
$\nrm{\chi_E}_{X_Q} \gtrsim 1.$
\item\label{it:jscor:2} $\BMO_X^* = \BMO_X$.
\end{enumerate}
Then we have
$\BMO_X^* \hookrightarrow \BMO$.
\end{theorem}
The assumption in Theorem \ref{jscor}\ref{it:jscor:1} is rather weak and easily checkable in concrete situations. However, in Example \ref{ex} we will show that it is not necessary for either the embedding $\BMO_X^* \hookrightarrow \BMO$ or the identity $\BMO_X^* = \BMO_X$. On the other hand, without any assumptions on $X$ the embedding $\BMO_X^* \hookrightarrow \BMO$ can fail, see Example \ref{ex:expweight}. We leave a full characterization of the embedding $\BMO_X^* \hookrightarrow \BMO$ as an open problem. In Proposition \ref{prop:charembed} we will characterize the weaker embedding $\BMO_X \hookrightarrow \BMO$ in terms of the boundedness of the mapping $f \mapsto \abs{f}$ on $\BMO_X$, which, as we will see, implies the sufficiency of the assumption in Theorem \ref{jscor}\ref{it:jscor:2} for the embedding $\BMO_X^* \hookrightarrow \BMO$.

We note that Theorems \ref{mr} and \ref{jscor} recover all previously known results in this area, as we will explain in Sections \ref{sec:suf} and \ref{sec:ex}. Moreover, when the conditions of both theorems hold, we can conclude that
\begin{align}\label{eq:allequal}
  \BMO = \BMO_X = \BMO_{X}^*.
\end{align}
A full characterization of this statement is left as an open problem. In particular, one may wonder whether the embedding $\BMO\hookrightarrow \BMO_X$ implies the embeddings $\BMO_X^* \hookrightarrow \BMO$ or $\BMO_X \hookrightarrow \BMO$? If this is the case, then \eqref{eq:allequal} is equivalent to any of the statements in Theorem~\ref{mr}

\bigskip
This paper is organized as follows:
\begin{itemize}
  \item After some preliminaries on dyadic cubes and Banach function spaces in Section \ref{sec:prelim}, we  will prove Theorems \ref{mr} and \ref{jscor} in Section~\ref{sec:main}. Moreover, we will discuss a surprising self-improvement result of the embedding $\BMO \hookrightarrow \BMO_X$.
  \item We will discuss some efficient and easy to use sufficient conditions to check the assumptions in Theorem \ref{mr} in Section \ref{sec:suf}.
  \item We will study our main results in several concrete situations in Section \ref{sec:ex}. In particular, we will study the case of weighted $L^1$-spaces, rearrangement invariant Banach function spaces, Orlicz spaces and variable exponent $L^p$-spaces.
  \item We state various open problems in Section \ref{sec:open}.
\end{itemize}

\section{Preliminaries}\label{sec:prelim}
\subsection{Dyadic cubes}\label{subsec:sparse} Denote by ${\mathcal Q}$ the set of all cubes $Q\subset {\mathbb R}^n$ with sides parallel to the axes and denote the space of all locally integrable functions by $L^1_{\loc}(\R^n)$. Let $M$ denote the standard Hardy-Littlewood maximal operator, i.e. for $f \in L^1_{\loc}(\R^n)$ we set
$$
Mf := \sup_{Q \in \mc{Q}} \, \ip{\abs{f}}_Q \chi_Q.
$$
Given a cube $Q\in {\mathcal Q}$, denote by
${\mathcal D}(Q)$ the set of all dyadic cubes with respect to $Q$, that is, the cubes obtained by repeated subdivision of $Q$ and each of its descendants into $2^n$ congruent subcubes.

\begin{definition}
Let $Q\in {\mathcal Q}$ and $\eta \in (0,1)$. We say that a family of dyadic cubes ${\mathcal F}\subset {\mathcal D}(Q)$ is $\eta$-sparse if $|E_P|\ge \eta|P|$ for all $P\in {\mathcal F}$, where
$$E_P:=P\setminus \bigcup_{P'\in {\mathcal F}: P'\subset P}P'.$$
\end{definition}
We can always decompose an $\eta$-sparse family as ${\mathcal F}=\cup_{k=0}^{\infty}{\mathcal F}_k$,
where each ${\mathcal F}_k$ is a family of pairwise disjoint cubes, and for $\O_k=\cup_{P\in {\mathcal F}_k}P$ we have $\O_{k}\subset \O_{k-1}$. Moreover, we have
\begin{align*}
|\Omega_{k}|&=\sum_{P\in \Omega_{k-1}}|\Omega_{k}\cap P|=\sum_{P\in \Omega_{k-1}}\hab{|P|-|P\setminus \Omega_{k}|}\\&=\sum_{P\in \Omega_{k-1}}\hab{|P|-|P\setminus E_P|}
\le (1-\eta)\sum_{P\in \Omega_{k-1}}|P|=(1-\eta)|\Omega_{k-1}|,
\end{align*}
which in turn implies
\begin{equation}\label{omegak}
|\Omega_k|\le (1-\eta)^k|Q|.
\end{equation}

\subsection{(Quasi-)Banach function spaces}\label{sec:BFS}
Denote the space of measurable functions $f \colon \R^n \to \R$ by $L^0(\R^n)$.
An order ideal $X \subset L^0(\R^n)$ equipped with a (quasi-)norm $\nrm{\,\cdot\,}_X$ is called
a \emph{(quasi-)Banach function space} if it satisfies the following properties:
\begin{itemize}
  \item \textit{Compatibility:} If $f,g \in X$ with $\abs{f}\leq \abs{g}$, then $\nrm{f}_X\leq \nrm{g}_X$.
  \item \textit{Weak order unit:} There is an $f \in X$ with $f > 0$ a.e.
  \item \textit{Fatou property:} If $0\leq f_n \uparrow f$ and $\sup_{n\in \N}\nrm{f_n}_X<\infty$, then $f \in X$ and $\nrm{f}_X=\sup_{n\in\N}\nrm{f_n}_X$.
\end{itemize}

\begin{remark}
Our definition of a Banach function space coincides with the definition of a \emph{saturated Banach function space with the Fatou property} in \cite[Chapter 15]{Za67}, using $\R^n$ as measure space. Indeed, saturation as defined in \cite[Section 67]{Za67} is implied by the existence of a weak order unit through \cite[Theorem 67.1]{Za67} and the converse follows  from \cite[Theorem 67.4(b)]{Za67}.
Moreover, we note that the Fatou property ensures that $X$ is complete, see \cite[Theorem 65.1]{Za67}.
\end{remark}

If $X$ is a Banach function space, we define its associate space $X'$ as the space of all $g \in L^0(\R^n)$ such that
\begin{equation*}
\nrm{g}_{X'}:= \sup_{\nrm{f}_X \leq 1} \int_{\R^n} \abs{fg}<\infty,
\end{equation*}
which is again a Banach function space, see \cite[Theorem 68.1 and Theorem 71.4(b)]{Za67}.
 Moreover, by the Lorentz--Luxembourg theorem (see \cite[Theorem 71.1]{Za67}) we have $X''=X$, so in particular
\begin{equation}\label{eq:LL}
\nrm{f}_{X}= \sup_{\nrm{g}_{X'} \leq 1} \int_{\R^n} \abs{fg}.
\end{equation}

For any $f \in L^0(\R^n)$ we define its left-continuous non-increasing rearrangement by
$$f^*(t)=\inf\big\{\a>0:|\{x\in {\mathbb R}^n:|f(x)|>\a\}|< t\big\},\qquad t >0.$$
We call two functions $f,g \in L^0(\R^n)$ \emph{equimeasurable} if $f^*=g^*$. A quasi-Banach function space is called $X$ \emph{rearrangement invariant} if
for any equimeasurable $f,g \in X$  we have
$
\nrm{f}_X = \nrm{g}_X.
$

A Banach function space $X$ is called \emph{$q$-concave} for $q \in [1,\infty)$ if there is a constant $C>0$ such that for $f_1,\ldots,f_m\in X$ we have
\begin{equation*}
   \has{\sum_{k=1}^m \nrm{f_k}_X^q}^{1/q} \leq C\, \nrms{\has{\sum_{k=1}^m\abs{f_k}^q}^{1/q}}_{X}.
\end{equation*}
By renorming, we may assume without loss of generality that $C=1$ (see \cite[Theorem 1.d.8]{LT79}).

\section{Main results}\label{sec:main}
In this section we will prove our main results, i.e. Theorems \ref{mr} and \ref{jscor} from the introduction.

\subsection{The embedding \texorpdfstring{$\BMO \hookrightarrow \BMO_X$}{BMO -> BMOX}}
The proof of Theorem \ref{mr} is based on two key ingredients. The first is a sparse domination estimate for the oscillation $|f-\langle f\rangle_Q|$, which can be found in, e.g., \cite[Lemma 3.1.2]{Hyt18} or \cite[Proposition 5.4]{LLO21}.

\begin{prop}\label{spes}
Let $f \in L^1_{\loc}(\R^n)$. For any cube $Q$ and $\eta \in (0,1)$
there exists an $\eta$-sparse family ${\mathcal F}\subset \mc{D}(Q)$ such that for a.e. $x\in Q$,
\begin{equation*}
|f-\langle f\rangle_Q| \chi_Q(x)\lesssim \sum_{P\in {\mathcal F}}\left(\frac{1}{\abs{P}}\int_P\abs{f - \langle f\rangle_P}\right)\chi_P(x).
\end{equation*}
\end{prop}
The second ingredient is the following well known result of Coifman--Rochberg \cite{CR80}.
\begin{prop}\label{cr} Let $f \in L^1_{\loc}(\R^n)$ such that $Mf<\infty$ a.e. Then $\log(Mf)\in \BMO$ with
$$
\|\log(Mf)\|_{\BMO}\lesssim 1.
$$
\end{prop}

Using Propositions \ref{spes} and \ref{cr}, we can now prove the following, more general version of
Theorem \ref{mr}.

\begin{theorem}\label{mrfull}
Let $X = \cbrace{X_Q}$ be a family of normalized quasi-Banach function spaces and let $\eta,\gamma \in (0,1)$. The following statements are equivalent:
\begin{enumerate}[(i)]
\item\label{it:1} $\BMO \hookrightarrow \BMO_X$, i.e. there is a constant $C_1>0$ such that for $f \in \BMO$
$$
\nrm{f}_{\BMO_X} \leq C_1 \nrm{f}_{\BMO}
$$
\item\label{it:2} There is a constant $C_2>0$ such that for every cube $Q\in \mc{Q}$ and any $\eta$-sparse family ${\mathcal F}\subset {\mathcal D}(Q)$  we have
$$\Big\|\sum_{P\in {\mathcal F}}\chi_P\Big\|_{X_Q}\leq C_2.$$
\item\label{it:3} There is a constant $C_3>0$ such that  for every cube $Q \in \mc{Q}$ and any sequence of nested sets $\O_k\subset Q$ satisfying $|\O_k|\le \ga^k|Q|$ we have
\begin{equation*}
\nrms{\sum_{k=0}^{\infty}\chi_{\O_k}}_{X_Q}\leq C_3.
\end{equation*}
\item\label{it:4} There is a constant $C_4>0$ such that for any $Q \in \mc{Q}$ and $f \in \exp L(Q)$ we have
$$
\nrm{f}_{X_Q} \leq C_4 \, \nrm{f}_{\exp L(Q)}.
$$
\end{enumerate}
Moreover, we have
$
C_1\simeq C_2\simeq C_3\simeq C_4.
$
\end{theorem}

\begin{proof} We will prove the implications
$$
\text{\ref{it:2}}\Rightarrow\text{\ref{it:1}}\Rightarrow\text{\ref{it:4}}\Rightarrow\text{\ref{it:3}}\Rightarrow\text{\ref{it:2}}.
$$
Let us start with the implication \ref{it:2}$\Rightarrow$\ref{it:1}, which is an immediate consequence of Proposition \ref{spes}. Indeed, let $Q \in \mc{Q}$ be an arbitrary cube. For $f\in \BMO$,  we obtain  by Proposition \ref{spes} that
there exists an $\eta$-sparse family ${\mathcal F}\subset \mc{D}(Q)$ such that for a.e. $x\in Q$
$$
|f-\langle f\rangle_Q| \chi_Q(x)\lesssim \|f\|_{\BMO}\sum_{P\in {\mathcal F}}\chi_P(x).
$$
From this and \ref{it:2} we obtain
\begin{equation*}
\|f-\langle f\rangle_Q\|_{X_Q}\lesssim \|f\|_{\BMO}\Big\|\sum_{P\in {\mathcal F}}\chi_P\Big\|_{X_Q}\leq C_2 \, \|f\|_{\BMO},
\end{equation*}
which proves \ref{it:1}.

For \ref{it:1}$\Rightarrow$\ref{it:4} fix $Q \in \mc{Q}$ and let $\varphi  \in \exp L(Q)$ with
$\nrm{\varphi}_{\exp L(Q)} =1$, i.e.
$$
\frac{1}{\abs{Q}}\int_Q e^{\abs{\f}} = 2
$$
Extend $\varphi$ by zero outside $Q$ and define $f \in L^1_{\loc}(\R^n)$ by $f := \log[M(\ee^{\abs{\varphi}})]$. Since $\ee^{\abs{\varphi}} \in L^1(Q)$ and $f \equiv 1$ on $\R^n \setminus Q$, we have $M(\ee^{\abs{\varphi}})<\infty$ a.e. Hence, by Proposition \ref{cr}, we have
$$\|f\|_{BMO}\lesssim 1.$$
Furthermore, using $\log(t) \lesssim t^{1/2}$ and Kolmogorov's inequality, we have
\begin{align*}
\int_Q f &\lesssim \int_Q[M(\ee^{\abs{\f}}\chi_Q)]^{1/2}+|Q|
\lesssim  \Big(\frac{1}{|Q|}\int_Qe^{\abs{\f}}\Big)^{1/2}|Q|+|Q|\lesssim |Q|.
\end{align*}
Combining this with the previous estimate and applying \ref{it:1} yields
\begin{align*}
\|\f\|_{X_Q}\le \|f\|_{X_Q}&\lesssim \|(f-\langle f\rangle_Q)\|_{X_Q}+\langle f\rangle_Q\\
&\leq C_1 \, \|f\|_{\BMO}+\langle f\rangle_Q\lesssim C_1,
\end{align*}
which proves \ref{it:4} by homogeneity.

Next, for \ref{it:4}$\Rightarrow$\ref{it:3} we fix a cube $Q \in \mc{Q}$ and a sequence of nested sets $\O_k\subset Q$ satisfying $|\O_k|\le \ga^k|Q|$ with $0<\ga<1$. Set
$$\f:=\sum_{k=0}^\infty\chi_{\Omega_k}=\sum_{k=0}^{\infty}(k+1)\chi_{\O_k\setminus \O_{k+1}}.$$
Denote $c_{\gamma}=\frac{1}{2}\log\frac{1}{\gamma}$.
Observe that we have
\begin{align*}
\int_Qe^{c_{\gamma}\f}&=\sum_{k=0}^{\infty}e^{c_{\gamma}(k+1)}|\O_k\setminus\O_{k+1}|\\
&\le e^{c_{\gamma}}\Big(\sum_{k=0}^{\infty}(e^{c_{\gamma}}\gamma)^k\Big)|Q|\le c_\gamma'|Q|,
\end{align*}
so $\varphi \in \exp L(Q)$ with $\nrm{\varphi}_{\exp L(Q)}$ only depending on $\gamma$. Thus by \ref{it:4}  we  conclude
\begin{equation*}
\Big\|\sum_{k=0}^{\infty}\chi_{\Omega_k}\Big\|_{X_Q}\leq C_4\, \nrm{\varphi}_{\exp L(Q)}\lesssim C_4.
\end{equation*}

Finally \ref{it:3}$\Rightarrow$\ref{it:2} with $\eta = 1-\gamma$ follows directly from \eqref{omegak}. The general case follows since we now have e.g. \ref{it:1}$\Leftrightarrow$\ref{it:3} and \ref{it:1} is independent of $\gamma$.
\end{proof}

\begin{remark}
Theorem \ref{mrfull} includes the classical John-Nirenberg inequality as a special case. Indeed, it is clear that the family of normalized Banach function spaces $\cbrace{\exp L(Q)}$ satisfies Theorem \ref{mrfull}\ref{it:4}. Therefore, Theorem \ref{mrfull} yields
$$
\nrm{f}_{\BMO_{\exp L}} \lesssim \nrm{f}_{\BMO},
$$
i.e. there is a $C>0$ such that for any $Q \in \mc{Q}$ we have
$$
\frac{1}{\abs{Q}}\int_Q \ee^{\frac{C}{\nrm{f}_{\BMO}}\cdot\abs{f-\ip{f}_Q}} \leq 2
$$
This immediately implies for any $\alpha>0$ that
\begin{align*}
\absb{\{x\in Q:|f(x)-\langle f\rangle_Q|>\a\}}\leq  2|Q|\,\ee^{-\frac{C}{\|f\|_{\BMO}}\a}.
\end{align*}
Moreover,  Theorem \ref{mrfull} proves that $\exp L$ is extremal in the following sense: If $\BMO \hookrightarrow \BMO_X$, then we must have  $\exp L(Q) \hookrightarrow X_Q$ for all $Q \in \mc{Q}$.
\end{remark}

\subsection{The space \texorpdfstring{$\BMO_{MX}$}{BMOMX}}
Before turning to the converse embedding $\BMO_X^* \hookrightarrow \BMO$, we will first briefly discuss a self-improvement result of the embedding $\BMO \hookrightarrow \BMO_X$.
Given a cube $Q$ and an $f \in L^1_{\loc}(\R^n)$, define the local dyadic maximal operator by
$$M_Qf:=\sup_{P\in {\mathcal D}(Q)}\langle|f|\rangle_P\chi_P.$$
Let $X = \cbrace{X_Q}$ be a family of normalized quasi-Banach function spaces and let us define $\BMO_{MX}$ as the space of all locally integrable $f$ such that
$$\|f\|_{\BMO_{MX}}:=\sup_Q\|M_Q(f-\langle f\rangle_Q)\|_{X_Q}.$$
Then we obviously have
\begin{equation*}
\|f\|_{\BMO}\le \|f\|_{\BMO_{MX}}.
\end{equation*}
On the other hand, it was observed in \cite[Remark 5.6]{LLO21} that a stronger variant of Proposition \ref{spes} holds with the left-hand side replaced by $M_Q(f-\langle f\rangle_Q)$.
Therefore, if the sparse condition in Theorem \ref{mrfull}\ref{it:2} holds, we obtain that
$$\|M_Q(f-\langle f\rangle_Q)\|_{X_Q}\lesssim \|f\|_{BMO},$$
which implies
$$\|f\|_{\BMO_{MX}}\lesssim \|f\|_{\BMO}.$$
Combining this with Theorem \ref{mrfull}, we obtain the following surprising self-improvement corollary.

\begin{cor}
Let $X = \cbrace{X_Q}$ be a family of normalized quasi-Banach function spaces. The following statements are equivalent:
\begin{enumerate}[(i)]
\item $\BMO \hookrightarrow \BMO_X$.
\item $\BMO = \BMO_{MX}$.
\end{enumerate}
\end{cor}

\subsection{The embedding \texorpdfstring{$\BMO_X^* \hookrightarrow \BMO$}{BMOX -> BMO}}We now turn to the converse embedding, for which it is, as discussed in the introduction, more natural to work with $\BMO_X^*$. To help with our future discussions, let us first introduce notation for the assumption in Theorem \ref{jscor}\ref{it:jscor:1}.

\begin{definition}\label{Adelta} Let $X=\{X_Q\}$ be a family of normalized quasi-Banach function spaces, and let $0<\d<1$. We say that $X$ satisfies the $\A{\d}$-condition if for every
cube $Q\in \mc{Q}$ and any measurable subset $E\subset Q$ with $|E|\ge \d|Q|$, we have
$ \|\chi_E\|_{X_Q}\gtrsim 1.$
\end{definition}

Note that the assumption in Theorem \ref{jscor}\ref{it:jscor:1} is exactly the $\A{1/2}$-condition. The $\A{\d}$-condition gets stronger as~$\d$ decreases. In general, the $\A{\d}$-condition does not necessarily imply the $\A{\varepsilon}$-condition for $\varepsilon<\d$. We will give an example of this in Subsection \ref{vlp}.

The sufficiency of the assumption in Theorem \ref{jscor}\ref{it:jscor:1} for the embedding $\BMO_X^* \hookrightarrow \BMO$ is an easy consequence of the John--Str\"omberg theorem \cite{J65, S79}.

\begin{theorem}\label{jscor1} Let $X=\{X_Q\}$ be a family of normalized quasi-Banach function spaces satisfying the $\A{1/2}$-condition. Then
$\BMO_X^* \hookrightarrow \BMO$.
\end{theorem}
\begin{proof} Let $Q \in \mc{Q}$. By the John--Str\"omberg theorem we have
\begin{equation}\label{js}
\|f\|_{\BMO}\lesssim \sup_Q\inf_c((f-c)\chi_Q)^*(|Q|/2).
\end{equation}
For a fixed scalar $c$, consider the set
$$E=\{x\in Q:|f(x)-c|\ge ((f-c)\chi_Q)^*(|Q|/2)\}.$$
Then $|E|\ge |Q|/2$. Therefore, by the $\A{1/2}$-condition, we have
$$((f-c)\chi_Q)^*(|Q|/2)\}\lesssim \|(f-c)\chi_E\|_{X_Q}\le \|f-c\|_{X_Q}.$$
Taking the infimum over $c$, the supremum over $Q$ and using \eqref{js} completes the proof.
\end{proof}

\begin{remark}\label{rem3}
An inspection of the proof of Theorem \ref{jscor1} shows that we actually used only two assumptions on $X=\{X_Q\}$: compatibility and the $\A{1/2}$-condition.
That is, we do not need completeness or quasi-normability in order to deduce the embedding $\BMO_X^* \hookrightarrow \BMO$.
\end{remark}

In order to prove the sufficiency of the assumption in Theorem \ref{jscor}\ref{it:jscor:2} for the embedding a characterization of the slightly weaker embedding $\BMO_X^* \hookrightarrow \BMO$, we will first give a characterization of the slightly weaker embedding $\BMO_X \hookrightarrow \BMO$ in terms of the boundedness of the mapping $f \mapsto \abs{f}$ on $\BMO_X$.

\begin{prop}\label{prop:charembed}
  Let $X = \cbrace{X_Q}$ be a family of normalized quasi-Banach function spaces. The following statements are equivalent:
  \begin{enumerate}[(i)]
    \item\label{it:charembed:1} $\BMO_X \hookrightarrow \BMO$.
    \item\label{it:charembed:2} For $f \in \BMO_X$ we have
    $
    \nrmb{\abs{f}}_{\BMO_X} \lesssim \nrm{f}_{\BMO_X}.
    $
  \end{enumerate}
\end{prop}

\begin{proof}
  Fix $f \in \BMO_X$. First assume that \ref{it:charembed:1} holds. Then we have for any $Q \in \mc{Q}$
  \begin{align*}
    \nrmb{\abs{f} - \ip{\abs{f}}_Q}_{X_Q} &\leq \nrms{\frac{1}{\abs{Q}} \int_Q \abs{f - f(x)}\dd x}_{X_Q}\\&\lesssim \nrmb{\abs{f - \ip{f}_Q}}_{X_Q}+\nrm{\chi_Q}_{X_Q} \nrm{f}_{\BMO}\lesssim \nrm{f}_{\BMO_X}
  \end{align*}
  Taking the supremum over all $Q \in \mc{Q}$ yields \ref{it:charembed:2}. Conversely, assume that \ref{it:charembed:2} holds and fix $Q \in \mc{Q}$. For any $g \in L^1_{\loc}(\R^n)$ we have
  \begin{align*}
  \ip{\abs{g}}_Q &\lesssim \ip{\abs{g}}_Q \nrm{\chi_Q}_{X_Q} \lesssim  \nrmb{\abs{g}-\ip{\abs{g}}_Q}_{X_Q} + \nrm{g}_{X_Q}
  \end{align*}
  Taking $g = {f - \ip{f}_Q}$, we obtain
  \begin{align*}
    \frac{1}{\abs{Q}}\int_Q\abs{f - \ip{f}_Q} &\lesssim \nrmb{\abs{g}-\ip{\abs{g}}_Q}_{X_Q} + \nrm{{f - \ip{f}_Q}}_{X_Q}\\
    &\lesssim \nrm{g}_{\BMO_X} + \nrm{f}_{\BMO_X}\leq 2\,\nrm{f}_{\BMO_X}.
  \end{align*}
  Taking the supremum over all $Q \in \mc{Q}$ yields \ref{it:charembed:1}, finishing the proof.
\end{proof}

It is easy to show that the mapping $f \mapsto \abs{f}$ is bounded on $\BMO_X^*$ for any family of normalized quasi-Banach function spaces $X = \cbrace{X_Q}$. Therefore, when $\BMO_X=\BMO_X^*$, we can deduce that $\BMO_X^* \hookrightarrow \BMO$ from Proposition \ref{prop:charembed}, finishing the proof of Theorem \ref{jscor}.
\begin{cor}\label{necbmo}
Let $X = \cbrace{X_Q}$ be a family of normalized quasi-Banach function spaces. If $\BMO_X=\BMO_X^*$, then $\BMO_X^*\hookrightarrow \BMO$.
\end{cor}

\begin{proof}
For $f \in \BMO_X^*$ and any $Q \in \mc{Q}$ we have
\begin{align*}
  \inf_c \,\nrmb{\abs{f}-c}_{X_Q} &\lesssim \nrmb{(x,y) \mapsto \abs{f(x)}-\abs{f(y)}}_{X_Q \times X_Q} \\&\leq \nrm{(x,y) \mapsto f(x)-f(y)}_{X_Q \times X_Q} \\ &\lesssim \inf_c \,\nrm{f-c}_{X_Q},
\end{align*}
so the mapping $f \mapsto \abs{f}$ is bounded on $\BMO_X^* = \BMO_X$. The corollary now follows from Proposition \ref{prop:charembed}.
\end{proof}

\section{Sufficient conditions}\label{sec:suf}
The $\A{1/2}$-assumption in Theorem \ref{jscor1} for the embedding $\BMO_X^*\hookrightarrow \BMO$ is rather weak and easily checked in concrete situations. The equivalent conditions in Theorem \ref{mrfull} for the embedding $\BMO\hookrightarrow \BMO_X$ are a bit more involved.
In this section we will explore some efficient and easy to use sufficient conditions that imply the conditions in Theorem \ref{mrfull} and thus yield the embedding $\BMO\hookrightarrow \BMO_X$.

\begin{definition}
   Given a family of normalized quasi-Banach function spaces $X=\{X_Q\}$, define the function $\psi_X\colon (0,1) \to \R_+$ by
$$\psi_X(t):=\sup_{Q \in \mc{Q}}\sup_{E\subset Q:|E|\le t|Q|}\|\chi_E\|_{X_Q}\qquad t \in (0,1).$$
\end{definition}

We can use the function $\psi_X$ to check the condition in Theorem \ref{mrfull}\ref{it:3}.

\begin{prop}\label{upb} Let $X=\{X_Q\}$ be a family of normalized Banach function spaces. Then
\begin{equation*}
\|f\|_{\BMO_X}\lesssim {\int_0^1\psi_X(t)\frac{dt}{t}}\cdot\|f\|_{\BMO}.
\end{equation*}
\end{prop}

 A similar statement can be proven for quasi-Banach function spaces, using the Aoki--Rolewicz theorem (see \cite{KPR84}).

\begin{proof} Let $Q$ be an arbitrary cube and let $\O_k\subset Q$  be a sequence of nested sets  satisfying $|\O_k|\le 2^{-k}|Q|$. Then we have
$$\Big\|\sum_{k=0}^{\infty}\chi_{\Omega_k}\Big\|_{X_Q}\le \sum_{k=0}^{\infty}\|\chi_{\Omega_k}\|_{X_Q}\le \sum_{k=0}^{\infty}\psi_X(2^{-k}).$$
It follows from the definition of $\psi_X$ that
$\psi_X(t)\le 2\,\psi_X(t/2).$
Using also that $\psi_X$ is monotone, we obtain
$$\sum_{k=0}^{\infty}\psi_X(2^{-k})\le 3\sum_{k=1}^{\infty}\psi_X(2^{-k})\leq 3 \int_0^1\psi_X(t)\frac{dt}{t}.$$
Combined with the previous estimates, this yields the condition in Theorem \ref{mrfull}\ref{it:3} and thus completes the proof.
\end{proof}

Proposition \ref{upb} states that
$$\int_0^1\psi_X(t)\frac{dt}{t}<\infty\qquad \Rightarrow \qquad \BMO\hookrightarrow \BMO_X.$$
The converse is implication is false. Indeed, for $X=\{\exp L(Q)\}$ the classical John--Nirenberg inequality states $\BMO\hookrightarrow \BMO_X$, but $\psi_X(t)\simeq \log(e/t)^{-1}$ and thus
$\int_0^1\psi_X(t)\frac{\ddn t}{t}= \infty$.

It turns out that the decay of $\psi_X$ for $X=\{\exp L(Q)\}$ is extremal in the following sense:

\begin{lemma}\label{converse} Let $X=\{X_Q\}$ be a family of normalized quasi-Banach function spaces.
If $\BMO\hookrightarrow \BMO_X$, then
$$\psi_X(t)\lesssim \log(e/t)^{-1}\qquad t \in (0,1).$$
\end{lemma}

\begin{proof}
Given $E\subset Q$, define
$$g:=\max\cbraces{\log\has{\frac{|Q|}{|E|}M\chi_E},0}.$$
Then, by Proposition \ref{cr}, $\|g\|_{\BMO}\lesssim 1$. Moreover, $g\ge \log\frac{|Q|}{|E|}$ for a.e. $x\in E$ and, by Kolmogorov's inequality,
$$\int_Qg\lesssim \has{\frac{|Q|}{|E|}}^{1/2}\int_Q(M\chi_E)^{1/2}\lesssim |Q|.$$
Combining these properties yields
\begin{align*}
\log\frac{|Q|}{|E|}\|\chi_E\|_{X_Q}\leq \|g\|_{X_Q}&\lesssim \|g-\langle g\rangle_Q\|_{X_Q}+\langle g\rangle_Q
\lesssim 1,
\end{align*}
from which the lemma follows.
\end{proof}

Lemma \ref{converse},  in particular, shows that if $\BMO\hookrightarrow \BMO_X$, then  $\lim_{t\to 0}\psi_X(t)=0.$ This in turn implies the $\A{1-\varepsilon}$-condition for some $\varepsilon>0$, as we will show next.

\begin{cor}\label{cor:decay}
Let $X=\{X_Q\}$ be a family of normalized quasi-Banach function spaces such that
$\lim_{t\to 0}\psi_X(t)=0.$
Then there exists $\e>0$ such that $X$ satisfies the $\A{1-\e}$-condition.
\end{cor}

\begin{proof}
Suppose that $E\subset Q$ with $|E|\ge (1-\e)|Q|$. Then
$$\nrm{\chi_Q}_{X_Q} \leq \|\chi_E\|_{X_Q}+\|\chi_{Q\setminus E}\|_{X_Q}\leq \|\chi_E\|_{X_Q}+\psi_X(\e).$$
Hence, taking $\e>0$ such that $\psi_X(\e) \leq \frac12\nrm{\chi_Q}_{X_Q}$, we obtain the $\A{1-\e}$ condition.
\end{proof}

Note that the $\A{1-\e}$-condition obtained in Corollary \ref{cor:decay}  is not enough to conclude $\BMO_{X}^* \hookrightarrow \BMO$ by Theorem \ref{jscor1}, as that would require the $\A{1/2}$-condition. However, the $\A{1-\e}$-condition is the best we can deduce from decay of $\psi_X$ as $t \to 0$, as we will show in Example \ref{ex}.

\bigskip

Let $X$ be a Banach function space and consider the family $X=\{X_Q\}$, where
$\|f\|_{X_Q}:=\frac{\|f\chi_Q\|_X}{\|\chi_Q\|_X}.$
As mentioned in the introduction, the previous works
\cite{H09, H12, I16, INS19} establish the embedding $\BMO \hookrightarrow \BMO_X$ under the assumption that~$M$ is bounded on $X'$. We end this section by showing that this also follows from Theorem \ref{mrfull}. Moreover, we also establish the embedding $\BMO \hookrightarrow \BMO_X$ under the assumption that $M$ is restricted weak type bounded on $X$, which in particular includes the result on variable Lebesgue spaces in \cite{IST14} (see Subsection \ref{vlp}).

\begin{prop}\label{prop:weakmax}
  Let $X$ be a Banach function space. Assume either of the following conditions:
  \begin{enumerate}[(i)]
  \item\label{itsufprop1} $M$ is bounded on $X'$.
  \item\label{itsufprop2}  $X$ is $q$-concave for some $q<\infty$ and $M$ is restricted weak type bounded on $X$, i.e. for any measurable subset $E \subset \R^n$ we have
  $$
  \sup_{\lambda>0} \lambda\, \nrm{\chi_{\cbrace{M\chi_E>\lambda}}}_X \lesssim \nrm{\chi_E}_X.
  $$
  \end{enumerate}
  Then we have $\BMO = \BMO_X = \BMO_X^*$.
\end{prop}

\begin{proof}
For \ref{itsufprop1} fix a cube $Q \in \mc{Q}$ and let ${\mathcal F}\subset {\mathcal D}(Q)$ be an $\frac12$-sparse family. Then, for any $g \in X'$, we have
\begin{align*}
  \int_{\R^n}\Big(\sum_{P\in {\mathcal F}}\chi_P\Big)\abs{g}=\sum_{P\in {\mathcal F}}\int_P\abs{g}&\leq 2 \sum_{P\in {\mathcal F}}\int_{E_P}Mg=2\int_QMg\\
  &\leq  2\,\|Mg\|_{X'}\|\chi_Q\|_{X}\lesssim \|g\|_{X'}\|\chi_Q\|_{X}.
\end{align*}
By \eqref{eq:LL}, this implies
$$\Big\|\sum_{P\in {\mathcal F}}\chi_P\Big\|_{X}\lesssim \|\chi_Q\|_{X},$$
so Theorem \ref{mrfull}\ref{it:2} holds and thus $\BMO \hookrightarrow\BMO_X$. Moreover, for any measurable $E\subset Q$ with $\abs{E} \geq \frac12 \abs{Q}$ we have by \eqref{eq:LL}
\begin{align*}
  \nrm{\chi_{Q}}_X = \sup_{\nrm{g}_{X'}\leq 1} \int_Q \abs{g} &\leq 2\, \sup_{\nrm{g}_{X'}\leq 1} \int_E Mg \\&\leq 2\, \nrm{\chi_E}_X \sup_{\nrm{g}_{X'}\leq 1} \nrm{Mg}_{X'} \lesssim \nrm{\chi_E}_X,
\end{align*}
so $\cbrace{X_Q}$ satisfies the $\A{1/2}$-condition, which yields $\BMO_X^* \hookrightarrow\BMO$ by Theorem \ref{jscor1}.

 For \ref{itsufprop2} again fix a cube $Q \in \mc{Q}$ and take $E\subset Q$ with $\abs{E} \geq \frac12\abs{Q}$. Then we have
  \begin{equation*}
  \nrm{\chi_Q}_X \leq \nrm{\chi_{\cbrace{M\chi_E\geq\frac12}}}_X \lesssim \nrm{\chi_E}_X,
  \end{equation*}
  so $X$ satisfies the $\A{1/2}$-condition. By Theorem \ref{jscor1}, we therefore have $\BMO_X^* \hookrightarrow \BMO$.
   It remains to show $\BMO \hookrightarrow \BMO_X$, for which we will check the condition in Theorem \ref{mrfull}\ref{it:2}. Fix a cube $Q \in \mc{Q}$ and let $\mc{F} \subset \mc{D}(Q)$ be a $\frac12$-sparse family of cubes. Write $\mc{F} = \bigcup_{k=0}^\infty \mc{F}_k$, where each ${\mathcal F}_k$ is a family of pairwise disjoint cubes, and for $\O_k=\cup_{P\in {\mathcal F}_k}P$ we have $\O_{k}\subset \O_{k-1}$. Fix $k \in \N \cup \cbrace{0}$. Since $M$ is restricted weak type bounded on $X$, there is a $C>0$ such that
    \begin{align*}
     \nrms{\sum_{P \in \mc{F}_k}\chi_P}_X &\leq \nrmb{\chi_{\cbrace{M(\sum_{P \in \mc{F}_k} \chi_{E_P}) \geq \frac12}}}_X
     \leq 2C \,\nrmb{\sum_{P \in \mc{F}_k}\chi_{E_P}}_X.
   \end{align*}
   Therefore, using the $q$-concavity of $X$, we have
   \begin{align*}
     \nrmb{\sum_{P \in \mc{F}_k}\chi_P}_X^q &\geq \nrmb{\sum_{P \in \mc{F}_k}\chi_{P\setminus E_P}}_X^q +\nrmb{\sum_{P \in \mc{F}_k}\chi_{E_P}}_X^q \\
     &\geq \nrmb{\sum_{P \in \mc{F}_{k+1}}\chi_{P}}_X^q + \frac{1}{(2C)^q}  \nrmb{\sum_{P \in \mc{F}_k}\chi_P}_X^q.
   \end{align*}
   Defining $\alpha := \ha{1-{(2C)^{-q}}}^{1/q}<1$, we have
   $$
   \nrmb{\sum_{P \in \mc{F}_{k+1}}\chi_{P}}_X \leq \alpha \, \nrmb{\sum_{P \in \mc{F}_k}\chi_P}_X,
   $$
   and thus
   $$
   \nrmb{\sum_{P \in \mc{F}_{k}}\chi_{P}}_X \leq \alpha^k \, \nrm{\chi_Q}_X.
   $$
  We conclude
   \begin{align*}
     \nrms{\sum_{P \in \mc{F}}\chi_P}_X &\lesssim \sum_{k=0}^\infty\nrmb{\sum_{P \in \mc{F}_k}\chi_P}_X\leq {\sum_{k=0}^\infty \alpha^{k}} \cdot \nrm{\chi_Q}_X
     \lesssim \nrm{\chi_Q}_X,
   \end{align*}
   which implies  $\BMO \hookrightarrow \BMO_X$ by Theorem \ref{mrfull}.
\end{proof}

\begin{remark}
Note that we did not need the full $q$-concavity estimate in the proof of Proposition \ref{prop:weakmax}\ref{itsufprop2}. Indeed, we only used
\begin{equation*}
  \nrm{\chi_{E\cup E'}}_X \geq \hab{\nrm{\chi_E}_X^q + \nrm{\chi_{E'}}_X^q}^{1/q}.
\end{equation*}
for measurable, disjoint $E,E'\subset \R^n$.
\end{remark}

\section{Examples}\label{sec:ex}
We will now turn to concrete families of normalized quasi-Banach function spaces to which our main results are applicable. On the one hand we will show that our results generalize various previously known results in the literature. On the other hand, we will provide examples that will, in particular, show the following:
\begin{itemize}
  \item The $\A{\delta}$-condition is not necessary for the embedding $\BMO_X^* \hookrightarrow \BMO$ for any $\delta \in (0,1)$, see Example \ref{ex}. In particular, the $\A{1/2}$-condition assumed in Theorem \ref{jscor1} is not necessary.
  \item There exist $X$ such that $\BMO_X^* \not\hookrightarrow \BMO$, see Example~\ref{ex:expweight}.
  \item The conditions in Proposition \ref{prop:weakmax} are not necessary for either $\BMO \hookrightarrow \BMO_{X}$ or $\BMO_{X}^*\hookrightarrow \BMO$, see Remark \ref{rem:A1}, Subsection \ref{subs:riBFS} and Example \ref{ex}.
\end{itemize}

\subsection{Weighted \texorpdfstring{$L^1$}{L1}-spaces} We start by considering the case
$$
\nrm{f}_{X_Q} := \frac{\nrm{f\chi_Q}_{L^1(w)}}{\nrm{\chi_Q}_{L^1(w)}} = \frac{1}{w(Q)} \int_Q \abs{f}w
$$
for a weight $w$, i.e. a locally integrable $w \colon \R^n \to (0,\infty)$. In this setting the condition in Theorem \ref{mrfull}\ref{it:3} asserts that for any cube $Q \in \mc{Q}$ and any $\frac12$-sparse collection of cubes $\mc{F}\subseteq \mc{D}(Q)$ we have
$$
\sum_{P\in \mc{F}} w(P) \lesssim w(Q).
$$
This condition is equivalent to $w \in A_\infty$, where we say that $w \in A_\infty$ if the Fuji--Wilson $A_\infty$-constant
$$
[w]_{A_\infty}:= \sup_{Q \in \mc{Q}}\frac{1}{w(Q)} \int_Q M(w\chi_Q)
$$
is finite.
\begin{lemma}\label{lemma:ainfty}
  For any weight $w$ we have
  \begin{equation*}
    [w]_{A_\infty} \simeq \sup_{Q \in \mc{Q}} \sup_{\mc{F} \subset \mc{D}(Q)} \frac{\sum_{P\in \mc{F}} w(P)}{w(Q)},
  \end{equation*}
  where the second supremum is taken over all $\frac12$-sparse collections $\mc{F}$.
\end{lemma}

\begin{proof}
  Fix a cube $Q\in \mc{Q}$. First suppose that $w \in A_\infty$ and take a $\frac{1}{2}$-sparse family of cubes $\mc{F} \subset \mc{D}(Q)$. Then we have
  \begin{align*}
    \sum_{P\in \mc{F}} w(P)\leq 2\sum_{P \in \mc{F}}\int_{E_P} M(w\chi_P)\leq 2 \int_{Q} M(w\chi_Q) \leq 2 \,[w]_{A_\infty} \,w(Q).
  \end{align*}
  Conversely, by \cite[Lemma 2.6]{Le17} we can find a $\frac12$-sparse collection of cubes $\mc{F} \subset \mc{D}(Q)$ such that
  $$
  M(w\chi_Q)(x) \lesssim \sum_{P \in \mc{F}} \ip{w}_P \chi_P(x), \qquad x \in Q.
  $$
  Therefore, we have
  $$
  \int_{Q} M(w\chi_Q) \lesssim \int_Q \sum_{P \in \mc{F}} \ip{w}_Q \chi_P = \sum_{P\in \mc{F}} w(P),
  $$
  finishing the proof.
\end{proof}

As a direct corollary of Lemma \ref{lemma:ainfty}, Theorem \ref{mrfull} recovers the recent characterization of the Fuji--Wilson $A_\infty$-constant in \cite[Corollary 2.1]{OPRR19}.
\begin{cor}\label{cor:Ainfty}
  For any weight $w$ we have
  \begin{equation*}
    [w]_{A_\infty}\simeq \sup_{\nrm{f}_{\BMO}\leq 1} \nrm{f}_{\BMO_{L^1(w)}}.
  \end{equation*}
\end{cor}
\begin{remark}
  Note that the argument in Lemma \ref{lemma:ainfty} could also be applied to $X_Q = \frac{\nrm{f\chi_Q}_{L^1(w)}}{Y(Q)}$ for any functional $Y \colon \mc{Q} \to (0,\infty)$, adapting the definition of $A_\infty$ accordingly. This would recover the main theorem of \cite{OPRR19}.
\end{remark}

In the setting of weighted $L^1$-spaces,  the $\A{\delta}$-condition for $\delta \in (0,1)$ asserts that for any $Q \in \mc{Q}$ and any measurable $E \subseteq Q$ with $\abs{E} \geq \delta \abs{Q}$ we have $$w(E) \gtrsim w(Q).$$ This is equivalent to $w \in A_\infty$ by \cite[Lemma 5]{CF74}. So $w \in A_\infty$ implies that $\BMO_X^* \hookrightarrow \BMO$ by Theorem \ref{jscor}, which combined with Corollary \ref{cor:Ainfty} yields the classical result of Muckenhoupt--Wheeden \cite{MW75} that $$\BMO = \BMO_{L^1(w)} = \BMO_{L^1(w)}^*, \qquad w \in A_\infty.$$
\begin{remark}\label{rem:A1}
  Note that the sufficient conditions in Proposition \ref{prop:weakmax} would yield much more stringent conditions on $w$. In particular, Proposition \ref{prop:weakmax}\ref{itsufprop2} with  $X = L^1(w)$ is equivalent to the assumption that $w \in A_1 \subsetneq A_\infty$.
\end{remark}

We end the discussion of our results in the context of weighted $L^1$-spaces with an example that shows, in particular, that the $A_\infty$-condition is not necessary for the embedding $\BMO_{L^1(w)} \hookrightarrow \BMO$ and that the embedding $\BMO_X^* \hookrightarrow \BMO$ can fail.
\begin{example}\label{ex:expweight}
  Take $n=1$ and define $w(x) = \ee^{x}$. Then we have
  \begin{enumerate}[(i)]
  \item \label{itweightedL1:1} $\BMO_{L^1(w)} \hookrightarrow \BMO$
  \item \label{itweightedL1:2} $\BMO_{L^1(w)}^* \not\hookrightarrow \BMO$
  \end{enumerate}
\end{example}

\begin{proof}
For \ref{itweightedL1:1} let $f\in \BMO_{L^1(w)}$ and fix $0<\lambda < \nrm{f}_{\BMO}$. Define the (non-empty) family of intervals  $$\mc{F}=\cbraces{I: \frac{1}{\abs{I}} \int_{I}\abs{f - \ip{f}_I}\ge \frac{3}{4}\lambda}.$$
Define $d=\inf_{I\in \mc{F}}|I|$. First suppose that $d \leq 2$ and take an interval $I \in \mc{A}$ with $\abs{I} \leq 3$. Then, noting that
\begin{equation}\label{e^x}
\sup_{y\in J} w(y)\simeq \inf_{y\in J} w(y), \qquad \abs{J} \leq 3,
\end{equation}
we find that
$$\frac34\lambda \leq \frac{1}{\abs{I}} \int_{I}\abs{f - \ip{f}_I} \lesssim \frac{1}{w(I)}\int_{I} |f-\ip{f}_I| w \leq \nrm{f}_{\BMO_{L^1(w)}}.X$$

Now, assume that $d\ge 2$. By definition there exists an interval $I=(a,b) \in \mc{F}$ with $h:=|I|\le d+\frac12$ such that   $$\frac{1}{\abs{I}} \int_{I}\abs{f - \ip{f}_I}\ge \frac{3}{4}\lambda.$$
Define $I_0= (a,b-1]$ and $I_1 = (b-1,b)$ and observe that $I = I_0 \cup I_1$.
Then we have
\begin{align*}
  \langle f\rangle_{I}&= \frac{h-1}{h}\langle f\rangle_{I_0} + \frac{1}{h}\langle f\rangle_{I_1} = \frac{h-1}{h}\langle f\rangle_{I_0} +  \frac1h\ip{ f}_{I} +\frac1h \int_{I_1} f-\langle f\rangle_{I},
\end{align*}
so,  rearranging the terms, we obtain
\begin{equation}\label{iden}
\langle f\rangle_{I}-\langle f\rangle_{I_0}=\frac{1}{h-1}\int_{I_1} |f-\langle f\rangle_{I}|
\end{equation}
We claim that  \begin{equation}\label{eq:gtr14}
  \int_{I_1} |f-\langle f\rangle_{I}|\geq \frac\lambda4.
\end{equation}
Indeed, suppose that this is not the case. Then, using \eqref{iden}, we have
\begin{align*}
  \frac{3}{4}\lambda &\leq \frac{1}{|I|}\int_{I} |f-\langle f\rangle_{I}|=\frac{1}{|I|}\int_{I_0} |f-\langle f\rangle_{I}|+\frac{1}{|I|}\int_{I_1} |f-\langle f\rangle_{I}|\\
  &\le \frac{1}{h}\int_{I_0} |f-\langle f\rangle_{I_0}|+\frac{h-1}{h} |\langle f\rangle_{I}-\langle f\rangle_{I_0}|+\frac{1}{h}\int_{I_1} |f-\langle f\rangle_{I}|\\
  &\le \frac{h-1}{h} \cdot \frac{1}{\abs{I_0}}\int_{I_0} |f-\langle f\rangle_{I_0}|+\frac{1}{h} \has{\frac\lambda4+\frac\lambda4},
\end{align*}
which implies
$$ \frac{1}{\abs{I_0}}\int_{I_0} |f-\langle f\rangle_{I_0}| \geq \frac{h}{h-1} \has{\frac{3}{4}\lambda-\frac{\lambda}{2h}} = \frac{3}{4}\lambda +\frac{\lambda}{4} \frac{1}{h-1} >\frac34\lambda.$$
We deduce that $I_0 \in \mc{F}$, which is a contradiction since $|I_0|\le d-\frac{1}{2}$. So \eqref{eq:gtr14} must hold.

Now, observe that $w(I)=e^{b}-e^{a} \le e^{b}$. Thus, combining \eqref{e^x} and \eqref{eq:gtr14}, we obtain
\begin{align*}
  \frac{\lambda}{4}&\le \int_{I_1} |f-\langle f\rangle_{I}|\lesssim \frac{\inf_{y\in I_1} w(y)}{e^{b}}  \int_{I^+} |f-\langle f\rangle_{I}|\\
  &\le \frac{1}{e^{b}}  \int_{I_1} |f-\langle f\rangle_{I}| w \le \frac{1}{w(I)}  \int_{I} |f-\langle f\rangle_{I}| w \leq \nrm{f}_{\BMO_{L^1(w)}}.
\end{align*}

COmbining the cases $d \leq 2$ and $d \geq 2$, we conclude $\nrm{f}_{\BMO_{L^1(w)}} \gtrsim \lambda$ for all $0<\lambda < \nrm{f}_{\BMO}$, which shows that $f \in \BMO$ and $\nrm{f}_{\BMO} \lesssim \nrm{f}_{\BMO_{L^1(w)}}$.

\bigskip

  For \ref{itweightedL1:2} define $f(x) = x$. It is clear that $f \notin \BMO$. Since $f$ is Lipschitz continuous, we have
  $$
  \sup_{\abs{I} \leq 1} \inf_c \frac{1}{w(I)} \int_I\abs{f-c}w <\infty.
  $$
  Let $I=(a,b)$ be an interval with $\abs{I} = b-a >1$. Then
\begin{align*}
  \int_{a}^{b} \absb{x-b} \ee^{x}\dd x &=  \int_{a}^{b} (b-x) \ee^{x}\dd x =\ee^{b} \int_{0}^{b-a} y \ee^{-y}\dd y \simeq \ee^{b} \simeq w(I),
\end{align*}
which implies $$\inf_c \frac{1}{w(I)}\int_I \absb{x-c} \ee^{x}\dd x\lesssim 1.$$ We conclude that   $f\in \BMO^{*}_{L^{1}(w)}$, finishing the proof.
\end{proof}

\subsection{Rearrangement invariant Banach function spaces}\label{subs:riBFS}
Next, we  study the case of rearrangement invariant quasi-Banach function spaces. In this situation $\A\delta$-condition holds for any $\delta \in (0,1)$, which follows by subdividing $Q$ into sets of equal measure and using equimeasurability. Therefore, as a direct corollary to Theorem \ref{jscor1}, we have:
\begin{cor}\label{ricase}
Let $X=\{X_Q\}$ be a family normalized rearrangement invariant quasi-Banach function spaces. Then $\BMO_X^* \hookrightarrow \BMO$.
\end{cor}

For the converse embedding, i.e. $\BMO \hookrightarrow \BMO_X$, we note that the function $\psi_X$ introduced in Section \ref{sec:suf} simplifies significantly when
$X$ is a rearrangement invariant Banach function space and $\{X_Q\}$ for $Q \in \mc{Q}$ is given by
\begin{equation*}
\|f\|_{X_Q}:=\frac{\|f\chi_Q\|_{X}}{\|\chi_Q\|_X}.
\end{equation*}
Indeed, for $t >0 $ and $E\subset \R^n$ with $\abs{E}=t$ define
\begin{equation*}
\varphi_X(t):= \nrm{\chi_E}_X,
\end{equation*}
which by equimeasurability is independent of the chosen $E$. The function $\varphi_X$ so-defined is called the \emph{fundamental function} of $X$ (see e.g. \cite[Section 2.5]{BS88}). The function $\psi_X$ is the smallest submultiplicative majorant of the fundamental function of $X$ for $t \in (0,1)$, i.e.
\begin{equation*}
\psi_X(t)= \sup_{s >0} \frac{\varphi_X(st)}{\varphi_X(s)}, \qquad t \in (0,1).
\end{equation*}
In \cite{Sh70} it was shown that the boundedness of the maximal operator $M$ on either $X$ or $X'$ can not be characterized in terms of the fundamental function of $X$. In particular, the author constructs a rearrangement invariant Banach function space $X$  such that
$\varphi_X(t) = t^{1/2}$
and $M$ is unbounded on either $X$ or $X'$.
For these spaces we have
\begin{equation*}
\int_0^1\psi_X(t)\frac{dt}{t} = 2<\infty,
\end{equation*}
so  Proposition \ref{upb} and Corollary \ref{ricase} yield $\BMO = \BMO_X = \BMO_X^*$. This shows that the assumption in Proposition \ref{prop:weakmax}\ref{itsufprop1}, and therefore the assumption in the main results of \cite{H09, H12, I16, INS19}, is not necessary.

\subsection{Orlicz spaces}
We now turn to normalized Orlicz spaces, i.e. we will study $\BMO_{\f}$, which is defined using
$$
\nrm{f}_{\f,Q}:=  \inf \cbraces{\alpha>0 : \frac{1}{\abs{Q}} \int_Q \varphi\has{\frac{\abs{f}}{\a}}\dd x \leq 1}.
$$
In the recent paper \cite{CPR20}, the authors considered
a question about minimal assumptions on $\f$ for which $\BMO_{\f}\hookrightarrow \BMO$, see also \cite{LSSVZ15} for closely related results.
The main result in \cite{CPR20} was established in two stages:
\begin{enumerate}[(i)]
\item\label{itOrlicz:1} First, the embedding $\BMO_{\f}^*\hookrightarrow \BMO$ was shown assuming that $\f$ is increasing, concave, $\f(0)=0$ and $\lim_{t\to \infty}\f(t)=\infty$.
\item\label{itOrlicz:2}  Second, using the result in \ref{itOrlicz:1}, the monotonicity and concavity assumptions were relaxed to $\f$ being measurable.
\end{enumerate}
We note that, under less restrictive assumptions on $\f$, \ref{itOrlicz:1}  follows from Remark \ref{rem3}. Indeed, if $\f$ is non-decreasing, then the norm $\|\cdot\|_{\f,Q}$ is compatible. Further, if $\lim_{t\to 0}\f(t)=\f(0)=0$ and $\lim_{t\to \infty}\f(t)=\infty$,
then we have that
$$\|\chi_Q\|_{\f,Q}=\inf\{\la>0:\f(1/\la)\le 1\}<\infty,$$
and, for $E\subset Q$ with $|E|=|Q|/2$, we have
$$\|\chi_E\|_{\f,Q}=\inf\{\la>0:\f(1/\la)\le 2\}>0.$$
Thus, the $\A{1/2}$-condition holds, which by Remark \ref{rem3} implies that $\BMO_{\f}^*\hookrightarrow \BMO$.

Notice that we did not use the concavity assumption to deduce \ref{itOrlicz:1}. For this reason, \ref{itOrlicz:2} also follows at once. Indeed, assuming that $\psi$ is measurable,
$\psi(0)=0$ and $\lim_{t\to \infty}\psi(t)=\infty$, define
$$\f(t):=\inf_{x\in [t,\infty)}\min\cbrace{\psi(x), x}.$$
We obtain that $\f$ is non-decreasing, $\lim_{t\to 0}\f(t)=\f(0)=0$ and $\lim_{t\to \infty}\f(t)=\infty$. Moreover $\f\le \psi$. Therefore,
$$\|f\|_{\BMO}\lesssim \|f\|_{\BMO_{\f}^*}\le \|f\|_{\BMO_{\psi}^*}.$$

\bigskip

In
 a recent paper \cite{MRR20}, the embedding $\BMO \hookrightarrow \BMO_X$
has been obtained in different terms, which is difficult to compare with e.g. Theorem~\ref{mrfull} or Proposition \ref{upb} in general. Observe, however, that in some particular cases Proposition \ref{upb} provides a better result. For example, it was shown in \cite[Example 4.2]{MRR20} that
for $\f_{p,\a}(t):=t^p(1+\log^+(t))^{\a}$, one has
$$\|f\|_{\BMO_{\varphi_{p,\alpha}}}\lesssim 2^{\a}(p+\a+1)\|f\|_{\BMO}\qquad p\ge 1,\,\a>0.$$
Using Proposition \ref{upb}, we can show a more precise estimate:
\begin{example}
  Define $\f_{p,\a}(t):=t^p(1+\log^+(t))^{\a}$ for $p \in [1,\infty)$ and $\a>0$. Then we have $\BMO \hookrightarrow \BMO_{\varphi_{p,\alpha}}$ with
  \begin{equation*}
\|f\|_{\BMO_{\varphi_{p,\alpha}}}\lesssim (p+\a)\|f\|_{\BMO}.
\end{equation*}
\end{example}

\begin{proof}
We have
$$\Psi_{p,\a}(t):=\sup_Q\sup_{E\subset Q:|E|\le t|Q|}\|\chi_E\|_{\f_{p,\a},Q}=\frac{1}{\f_{p,\a}^{-1}(1/t)},$$
and hence
\begin{align*}
\int_0^1\Psi_{p,\a}(t)\frac{\ddn t}{t}&=\int_0^1\frac{1}{\f_{p,\a}^{-1}(1/t)}\frac{\ddn t}{t}=\int_1^{\infty}\frac{1}{\f_{p,\a}^{-1}(t)}\frac{\ddn t}{t}\\
&=\int_{1}^{\infty}\frac{\f_{p,\a}'(t)}{\f_{p,\a}(t)}\frac{\ddn t}{t}=\int_1^{\infty}\Big(p+\frac{\a}{1+\log t}\Big)\frac{dt}{t^2}\simeq p+\a.
\end{align*}
The result therefore follows from Proposition \ref{upb}.
\end{proof}

\subsection{Variable exponent \texorpdfstring{$L^p$}{Lp}-spaces}\label{vlp}
Let $p:{\mathbb R}^n\to [1,\infty)$ be a measurable function. Denote
by $L^{p(\cdot)}$ the space of functions $f\colon \R^n \to \R$ such
that
$$\|f\|_{L^{p(\cdot)}}:=\inf\left\{\la>0:\int_{{\mathbb
R}^n}|f(x)/\la|^{p(x)}dx\le 1\right\}<\infty.$$
Denote $p_-:= \essinf_{x\in\mathbb{R}^n} p(x)$ and $p_+:=
\esssup_{x\in\mathbb{R}^n} p(x).$ In this final example section we will study $\BMO_{L^{p(\cdot)}}$, which is defined using
$$
\nrm{f}_{L^{p(\cdot)}(Q)} := \frac{\|f\chi_Q\|_{L^{p(\cdot)}}}{\|\chi_Q\|_{L^{p(\cdot)}}}.
$$
\begin{remark}
  Note that one could also define $\BMO_{L^{p(\cdot)}(\frac{\ddn x}{\abs{Q}})}$ using
  $$
\nrm{f}_{L^{p(\cdot)}(\frac{\ddn x}{\abs{Q}})} := \inf\left\{\la>0:\frac{1}{\abs{Q}}\int_{Q}|f(x)/\la|^{p(x)}dx\le 1\right\}.
$$
However, in this case one trivially has $$\BMO_{L^{p_+}} \hookrightarrow \BMO_{L^{p(\cdot)}(\frac{\ddn x}{\abs{Q}})}\hookrightarrow \BMO_{L^{p_-}}$$ and thus $\BMO_{L^{p(\cdot)}(\frac{\ddn x}{\abs{Q}})} = \BMO$ for any $p(\cdot)$ with $p_+<\infty$.
\end{remark}

It was shown in \cite{IST14} that if $p_+<\infty$ and the maximal operator is of weak type on $L^{p(\cdot)}$, i.e.
\begin{equation*}
\sup_{\lambda>0}\lambda\,\|\chi_{\{Mf>\lambda\}}\|_{L^{p(\cdot)}}\lesssim \|f\|_{L^{p(\cdot)}},
\end{equation*}
then $\BMO_{L^{p(\cdot)}} = \BMO$. Since $L^{p(\cdot)}$ is $p_+$-concave, this result is a special case of Proposition \ref{prop:weakmax}.

The main result of this subsection is the construction of a variable exponent $L^p$-space that shows that the $\A{\delta}$-condition is not necessary for any of the following embeddings:
\begin{itemize}
\item $\BMO \hookrightarrow \BMO_{X}$
\item $\BMO_{X}^*\hookrightarrow \BMO$
\item $\BMO_{X} =\BMO_X^*$
\end{itemize}
In particular, this implies that the neither the (restricted) weak type boundedness of $M$ on $X$ nor the boundedness of $M$ on $X'$ is necessary for either embedding, since these assumptions imply the $\A{\delta}$-condition for all $\delta \in (0,1)$ (see the proof of Proposition \ref{prop:weakmax}). Moreover, the example shows that the $\A{\delta}$-condition does not imply the $\A{\varepsilon}$-condition for $\varepsilon<\delta$.

\begin{example}\label{ex}
Given $0<\e<\d<1$, there exists an exponent $p\colon \R \to [1,\infty)$ with $p_+<\infty$ such that
\begin{enumerate}[(i)]
\item\label{exit1} $L^{p(\cdot)} \in \A{\delta} \setminus \A{\varepsilon}$.
\item\label{exit2} $\psi_{L^{p(\cdot)}}(t)\lesssim t^{1/2}$.
\item\label{exit3} $\BMO_{L^{p(\cdot)}}^*\hookrightarrow \BMO$.
\end{enumerate}
In particular, $\BMO = \BMO_{L^{p(\cdot)}}=\BMO_{L^{p(\cdot)}}^*$.
\end{example}

\begin{proof}
Our example will be a modification of examples considered in \cite[Ex. 4.51]{CUF13} and \cite[Th. 5.3.4]{DHHR11}. Take $\rho>1$ such that $\rho\e<\d$.
Let $\f$ be a $C^{\infty}$-function supported in $[-\rho\e/2,\rho\e/2]$, $0\le \f\le 1$ and $\f=1$ on $[-\e/2,\e/2]$. Define
$$p(x):=1+\sum_{k=1}^{\infty}\f(x-k).$$

Let us start by showing that $L^{p(\cdot)}\in \A{\d}$. Let $I$ be an arbitrary interval and let $E\subset I$ with $|E|=\d|I|$. Our goal is to show that
$$
\|\chi_I\|_{L^{p(\cdot)}}\lesssim \|\chi_E\|_{L^{p(\cdot)}},
$$
which is equivalent to
\begin{equation}\label{eqform}
\int_I\left(\frac{1}{\|\chi_E\|_{L^{p(\cdot)}}}\right)^{p(x)}dx\lesssim 1.
\end{equation}

Suppose first that $|I|\le N$, where $N\ge 1$ will be chosen later on.
Assume that $|E|\le 1$. Arguing as above, since $|E|^{1/p_-(E)}\le \|\chi_E\|_{L^{p(\cdot)}}$ and
$p$ is uniformly Lipschitz continuous, for every $x\in I$ we have
\begin{align*}
\left(\frac{1}{\|\chi_E\|_{L^{p(\cdot)}}}\right)^{p(x)}\le \left(\frac{1}{|E|}\right)^{\frac{p(x)-p_-(E)}{p_-(E)}}\frac{1}{|E|}\le \left(\frac{1}{\d|I|}\right)^{\frac{c|I|}{p_-(E)}}\frac{1}{\d|I|}
\lesssim \frac{1}{|I|},
\end{align*}
where the implicit constant depends only on $\d$ and $N$. Integrating this over $I$ yields (\ref{eqform}).
If $|E|>1$, then $\|\chi_E\|_{L^{p(\cdot)}}\ge 1$ and hence (\ref{eqform}) obviously holds since $|I|\lesssim 1$.

Assume now that $|I|>N$. Observe that $\|\chi_I\|_{L^{p(\cdot)}}\simeq |I|$.
Now the idea is to show that for $N$ large enough, the set $E$ will necessarily have its portion of measure at least $\d'|I|$ on the set where $p=1$. Then we obtain that $\|\chi_E\|_{L^{p(\cdot)}}\simeq |I|$,
and so (\ref{eqform}) holds.

The worst situation is when $I$ is the minimal interval containing the supports of $N+1$ copies of $\f$. Then $|I|=N+\rho\e$. If $A=\{x\in I:p(x)=1\}$, then $|A|=N(1-\rho\e)$ and $|I\setminus A|=(N+1)\rho\e$.
Therefore, if $N$ is such that
$$\frac{(N+1)\rho\e}{N+\rho\e}<\d,$$
we obtain that $|E|/|I|\ge \d'>0$, and this completes the proof of (\ref{eqform}).

\bigskip

Next, let us show that $L^{p(\cdot)}\not\in \A{\e}$. For $m\in {\mathbb N}$ take $I=[\e/2,m+\e/2]$ and set
$$E=\cup_{k=1}^m[k-\e/2,k+\e/2].$$
Then $E\subset I$ and $|E|=\e|I|$. Since $p=2$ on $E$, we have
$$\|\chi_E\|_{L^{p(\cdot)}}=(\e m)^{1/2}.$$
Next, set $F=\cup_{k=0}^{m-1}[k+\rho\e/2,k+1-\rho\e/2].$ Then $F\subset I$, and since $p=1$ on~$F$, we have
$$(1-\rho\e)m=\|\chi_{F}\|_{L^{p(\cdot)}}\le \|\chi_{I}\|_{L^{p(\cdot)}}.$$
Therefore, the estimate $\|\chi_I\|_{L^{p(\cdot)}}\lesssim \|\chi_E\|_{L^{p(\cdot)}}$ does not hold, which proves that $L^{p(\cdot)}\not\in \A{\e}$.

\bigskip

We now turn to the proof of property \ref{exit2}. Suppose first that $|I|\le 1$. Set $p_-(I)=\essinf_Ip$. Since $|I|^{1/p_-(I)}\le \|\chi_I\|_{L^{p(\cdot)}}$ and
$p$ is uniformly Lipschitz continuous, for every $x\in I$,
\begin{align*}
\left(\frac{1}{\|\chi_I\|_{L^{p(\cdot)}}}\right)^{p(x)}\le \left(\frac{1}{|I|}\right)^{\frac{p(x)-p_-(I)}{p_-(I)}}\frac{1}{|I|}\le \left(\frac{1}{|I|}\right)^{\frac{c|I|}{p_-(I)}}\frac{1}{|I|}
\lesssim \frac{1}{|I|}.
\end{align*}
From this, for any subset $E\subset I$ with $|E|=t|I|$,
$$
\int_E \left(\frac{1}{t^{1/2}\|\chi_I\|_{L^{p(\cdot)}}}\right)^{p(x)}dx\le \frac{1}{t}\int_E \left(\frac{1}{\|\chi_I\|_{L^{p(\cdot)}}}\right)^{p(x)}dx\lesssim 1,
$$
which is equivalent to
$$\|\chi_E\|_{L^{p(\cdot)}}\lesssim t^{1/2}\|\chi_I\|_{L^{p(\cdot)}}.$$

Assume now that $|I|>1$. Let $E\subset I$ with $|E|=t|I|$. Suppose that $|E|\le 1$. Then
$$\|\chi_{E}\|_{L^{p(\cdot)}}\le |E|^{1/2}=t^{1/2}|I|^{1/2}\le t^{1/2}\|\chi_{I}\|_{L^{p(\cdot)}}.$$

It remains to consider the case when $|I|>1$ and $|E|>1$. Denote $A=\{x:p(x)=1\}$ and observe that
for every interval $I$ with $|I|>1$ we have $|I\cap A|\ge \frac{1-\rho\e}{1+\rho\e}|I|$ (the worst situation is when $I$ contains the supports of two adjacent copies of $\f$,
and in this case $|I|=1+\rho\e$ and $|I\cap A|=1-\rho\e$). Then
\begin{align*}
  \|\chi_{E}\|_{L^{p(\cdot)}}\le |E|\le \frac{1+\rho\e}{1-\rho\e}t|I\cap A|&=\frac{1+\rho\e}{1-\rho\e}t\|\chi_{I\cap A}\|_{L^{p(\cdot)}}\lesssim t\,\|\chi_I\|_{L^{p(\cdot)}}.
\end{align*}
This completes the proof of \ref{exit2}.

\bigskip

We finish the proof of the example by showing $\BMO_{L^{p(\cdot)}}^*\hookrightarrow \BMO$. By the John--Stromberg theorem, it suffices to show that for every interval $I$,
\begin{equation}\label{itsuf}
\inf_{c}((f-c)\chi_I)^*(|I|/2)\lesssim \|f\|_{\BMO_{L^{p(\cdot)}}^*}.
\end{equation}

First, since $p$ is uniformly Lipschitz continuous and $p_+<\infty$, by \cite[Cor. 3.18]{CUF13} we have that the weak type boundedness of $M$ on $L^{p(\cdot)}$ holds locally, i.e.
\begin{equation*}
\sup_{\lambda>0}\lambda\,\|\chi_{\{Mf>\lambda\}}\|_{L^{p(\cdot)}(I)}\lesssim \|f\|_{L^{p(\cdot)}(I)},
\end{equation*}
for every interval $I$ with $|I|\le r$ for $r>0$ In particular this means that
\begin{equation}\label{locpart}
\langle |f|\rangle_I\|\chi_I\|_{L^{p(\cdot)}}\lesssim \|f\chi_I\|_{L^{p(\cdot)}},\qquad|I|\le 7.
\end{equation}
From this, by Chebyshev's inequality, we obtain that (\ref{itsuf}) holds for every interval $I$ with $|I|\le 7$.

Suppose now that $|I|>7$. Fix an arbitrary constant $c$, and denote $g=|f-c|$. Consider the set
$$E=\{x\in I: g(x)\ge (g\chi_I)^*(|I|/2)\},$$
for which we have $|E|\ge |I|/2$.

Take $m\in {\mathbb N}$ such that $m\le |I|<m+1$. Let $I_1,\ldots,I_m$  be arbitrary pairwise disjoint intervals of length 1 contained in $I$.
Let $m_0$ be the number of intervals $I_k$ such that $|I_k\cap E|\ge \frac14$. Denote these intervals by $I_{k_1},\ldots,I_{k_{m_0}}$.
For the other $m-m_0$ intervals $I_k$ we have $|I_k\cap E|<\frac14$. Denote by $G$ the union of these intervals. Then $|G|=m-m_0$ and
$$\frac{m}{2}\le \frac{ |I|}{2}\le |E|\le |G\cap E|+|I|-|G|\le \frac{m-m_0}{4}+m_0+1,$$
from which be deduce $m_0\ge \frac{m-4}{3}$.

Let us now introduce the set
$$E'=\{x\in I: g(x)\ge \frac{1}{8}(g\chi_I)^*(|I|/2)\}$$
and consider the following two cases.

{\bf Case 1:} For every $i=1,\dots, m_0$ we have $|I_{k_i}\cap E'|\ge \frac{1+\delta}{2}$. Denote $A=\{x:p(x)=1\}$. Observe that for every interval $J$ with $|J|=1$ we have $|J\cap A|\ge 1-\delta$.
Therefore, for $i=1,\dots, m_0$, we have
$$
|I_{k_i}\cap E'\cap A|\geq \abs{I_{k_i} \cap E'}+\abs{I_{k_i} \cap A}-\abs{I_k}\geq \frac{1-\delta}{2},
$$
which implies, since $m \geq 7$, that
\begin{align*}
  |E'\cap A|\ge \sum_{i=1}^{m_0}|I_{k_i}\cap E'\cap A|\ge m_0\frac{1-\delta}{2}&\ge (m-4) \frac{1-\delta}{6}\\&\ge (m+1) \frac{1-\delta}{18}\ge|I| \cdot \frac{1-\delta}{18}
\end{align*}
Combining this with the fact that $\|\chi_I\|_{L^{p(\cdot)}}\simeq |I|$ yields
$$\|\chi_I\|_{L^{p(\cdot)}}\lesssim |E'\cap A| = \|\chi_{E'\cap A}\|_{L^{p(\cdot)}} \leq \|\chi_{E'}\|_{L^{p(\cdot)}}.$$
By the definition of $E'$, it follows from this that
$$
(g\chi_I)^*(|I|/2)\leq 8 \, \frac{\|g\chi_{E'}\|_{L^{p(\cdot)}}}{\|\chi_{E'}\|_{L^{p(\cdot)}}} \lesssim \frac{\|g\chi_I\|_{L^{p(\cdot)}}}{\|\chi_I\|_{L^{p(\cdot)}}}.
$$

{\bf Case 2:} There is an $i \in \cbrace{1,\ldots,m_0}$ such that $|I_{k_{i}}\cap(I\setminus E')|>\frac{1-\delta}{2}.$
Denote $J=I_{k_{i}}$ and $J'=I_{k_{i}}\cap(I\setminus E')$. Since $|J\cap E|\ge \frac{1}{4}$, we have by the definition of $E$
$$\langle g\rangle_J  \geq \frac{1}{4}\ip{g}_{J\cap E} \geq \frac{1}{4}\,(g\chi_I)^*(|I|/2)$$
and, by the definition of $E'$, we have
$$g(x)<\frac{1}{8}\,(g\chi_I)^*(|I|/2),\qquad x \in J'.$$
Combined, these estimates yield
\begin{equation}\label{oscg}
|g(x)-\langle g\rangle_J|\ge \frac{1}{8}\,(g\chi_I)^*(|I|/2),\quad x\in J'.
\end{equation}
Further, observe that for every measurable set $G$ with $|G|\simeq 1$ we have $\|\chi_G\|_{L^{p(\cdot)}}\simeq 1$. Therefore,
$$\|\chi_J\|_{L^{p(\cdot)}}\lesssim \|\chi_{J'}\|_{L^{p(\cdot)}},$$
which, along with (\ref{oscg}), yields
\begin{equation}\label{gchiI}
(g\chi_I)^*(|I|/2)\lesssim \frac{\|(g-\langle g\rangle_J)\chi_J\|_{L^{p(\cdot)}}}{\|\chi_J\|_{L^{p(\cdot)}}}.
\end{equation}

By \ref{exit2} and Proposition \ref{upb}, we know that $\BMO\hookrightarrow\BMO_{L^{p(\cdot)}}$. Therefore, combining Proposition \ref{spes} with Theorem \ref{mrfull}\ref{it:2}, we
obtain
$$\frac{\|(g-\langle g\rangle_J)\chi_J\|_{L^{p(\cdot)}}}{\|\chi_J\|_{L^{p(\cdot)}}}\lesssim \sup_{|P|\le 1}\frac{1}{|P|}\int_P|g-\langle g\rangle_P|.$$
Combining this with  the previous estimate, standard properties of mean oscillations and \eqref{locpart}, yields
\begin{align*}
  ((f-c)\chi_I)^*(|I|/2) &\lesssim \sup_{|P|\le 1}\frac{1}{|P|}\int_P|g-\langle g\rangle_P| \\&\lesssim  \sup_{|P|\le 1}\frac{1}{|P|}\int_P|f-\ip{f}_P|\\
  &\lesssim  \sup_{|P|\le 1} \inf_{c'}\frac{1}{|P|}\int_P|f-c'| \lesssim \nrm{f}_{\BMO^*_{L^{p(\cdot)}}}.
\end{align*}

Combining the two cases, we obtain
$$((f-c)\chi_I)^*(|I|/2)\lesssim\frac{\|(f-c)\chi_I\|_{L^{p(\cdot)}}}{\|\chi_I\|_{L^{p(\cdot)}}} +\nrm{f}_{\BMO^*_{L^{p(\cdot)}}}.$$
This implies (\ref{itsuf}) and therefore completes the proof.
\end{proof}

\section{Open questions}\label{sec:open}
In this final section we collect several open questions.

\begin{que}\label{ques1}
Is there a family normalized quasi-Banach function spaces $X=\{X_Q\}$ for which the embedding $\BMO_X\hookrightarrow \BMO$ fails?
\end{que}

Observe that by Proposition \ref{prop:charembed} the absolute value mapping cannot be bounded on $\BMO_X$ and by Corollary \ref{ricase} such an $X$ cannot be rearrangement-invariant. Moreover, for every concrete, non-rearrangement-invariant family $X=\{X_Q\}$
considered in this paper, the embedding $\BMO_X\hookrightarrow \BMO$ holds.

\bigskip

If the answer to Question \ref{ques1} is positive, i.e. the embedding $\BMO_X\hookrightarrow \BMO$  is nontrivial, we can ask the following question:

\begin{que}\label{ques2}
Does the embedding $\BMO\hookrightarrow \BMO_X$ imply the converse embedding $\BMO_X\hookrightarrow \BMO$?
\end{que}

By Theorem \ref{mrfull}, the embedding $\BMO\hookrightarrow \BMO_X$ provides a lot of information about $X=\{X_Q\}$, and the question is whether this information enough
to establish that $\BMO_X\hookrightarrow \BMO$. One may also consider a stronger version of this question, asking whether the embedding $\BMO\hookrightarrow \BMO_X$ implies $\BMO_X^*\hookrightarrow \BMO$.

\bigskip

In Example \ref{ex:expweight} we established that the embedding $\BMO_X^*\hookrightarrow \BMO$ can fail. It is therefore natural to ask the following question:

\begin{que}\label{ques3}
What are non-trivial necessary conditions for the embedding $\BMO_X^*\hookrightarrow \BMO$?
\end{que}

Theorem \ref{jscor} provides sufficient conditions for $\BMO_X^*\hookrightarrow \BMO$, and of course it is desirable to find a full characterization. However, curiously enough, we were not able to
find any non-trivial necessary condition. In particular, it is natural to guess that the embedding $\BMO_X^*\hookrightarrow \BMO$ should imply the doubling condition for $X$, i.e. $\|\chi_{2Q}\|_X\lesssim \|\chi_Q\|_X$ for any $Q \in \mc{Q}$.

\bigskip

In Corollary \ref{necbmo} we saw that the identity $\BMO_X=\BMO_X^*$ implies the embedding $\BMO_X^* \hookrightarrow \BMO$. A sufficient condition for the identity $\BMO_X=\BMO_X^*$ was given in \cite[Theorem E]{IST14}. This sufficient condition implies the $\A{\delta}$-condition for all $\delta \in (0,1)$ and is thus not necessary by Example \ref{ex}. Therefore, one may wonder if there are weaker sufficient conditions for the identity $\BMO_X=\BMO_X^*$.

\begin{que}\label{eqsp}
Can we characterize when the spaces $\BMO_X$ and $\BMO_X^*$ coincide?
\end{que}

\bibliographystyle{plain}
\bibliography{bmo}

\end{document}